\documentclass[10pt,a4paper]{article}
\usepackage[utf8]{inputenc}
\usepackage[left=1in,right=1in,top=1in,bottom=1in]{geometry}
\usepackage{setspace}
\usepackage{amsmath}
\usepackage{amsfonts}
\usepackage{amssymb}
\usepackage{mathtools}
\usepackage{amsthm}
\usepackage{hyperref}
\usepackage{algorithmic}
\usepackage[symbol]{footmisc}

\usepackage{tikz-cd}
\usepackage{tikz}
\usetikzlibrary{shapes.geometric, arrows}

\tikzstyle{startstop} = [rectangle, rounded corners, 
minimum width=3cm, 
minimum height=1cm,
text centered, 
draw=black]
\tikzstyle{starts} = [rectangle, rounded corners, 
minimum width=3cm, 
minimum height=1cm,
text centered,
text width=2cm,
draw=black]
\tikzstyle{process} = [rectangle, rounded corners,
minimum width=6cm, 
minimum height=1cm, 
text centered, 
text width=6cm, 
draw=black]
\tikzstyle{arrow} = [thick,->,>=stealth]

\newtheorem{theorem}{Theorem}[section]

\newtheorem{proposition}{Proposition}[section]
\newtheorem{corollary_p}{Corollary}[proposition]
\newtheorem{lemma}{Lemma}[section]

\newtheorem{definition}{Definition}[section]

\newtheorem{algorithm}{Algorithm}[section] 

\author{Sumit Suthar\footnote{Email address: sutharsumit@iisc.ac.in, sumitsuthar@live.in}, Soumyendu Raha\footnote{Email address: raha@iisc.ac.in}\\ Indian Institute of Science, Bengaluru.}

\title{Explicitly Constrained Stochastic Differential Equations on Manifolds.}
\date{July 27, 2023}
\begin{document}
\maketitle
\begin{abstract}
In this manuscript we consider Intrinsic Stochastic Differential Equations on manifolds and constrain it to a level set of a smooth function. Such type of constraints are known as explicit algebraic constraints. The system of differential equation and the algebraic constraints is, in combination, called the Stochastic Differential Algebraic Equations (SDAEs). We consider these equations on manifolds and present methods for computing the solution of SDAEs on manifolds.
\end{abstract}
 Ito Stochastic Differential Equations on Manifolds. Global Stochastic Analysis, Constrained Dynamics, Constrained Stochastic Differential Equations.

\section{Introduction.}
In many areas of physics and engineering such as data science, machine-learning, molecular dynamics, astronomy, robotics, etc., we find that the state space of the dynamical system is a manifold. When these systems are subject to noise and operation/performance related constraints (e.g. no-slip velocity constraint in physics, data manifold as constraint in the subject of data science, trajectory constraint in robotics, etc.), it is natural to model the dynamical system using constrained Stochastic Differential Equations (SDEs) on manifolds. In this article, we focus on SDEs on manifolds that are constrained to evolve on a level set of a map between two manifolds. We call these equations as \textit{explicit Stochastic Differential-Algebraic Equations (SDAEs)}.

In \cite{suthar2021explicit}, the system of explicit Stochastic Differential Algebraic Equations on Euclidean spaces is defined as
\begin{subequations}
\label{eq:SDAE1}
\begin{equation}
\label{eq:SDAE-SDE}
dx = f(x,u)dt + \sigma(x,u)dW_t,  
\end{equation}
\begin{equation}
\label{eq:SDAE-AE}
    h(x,u) = 0
\end{equation}
\end{subequations}
where $W_t = (W^1_t, W^2_t, ..., W^d_t)$ is a d-dimensional Wiener process, $f:\mathbb{R}^n\times \mathbb{R}^m\to \mathbb{R}^n$, and $\sigma:\mathbb{R}^n\times \mathbb{R}^m\to L(\mathbb{R}^d,\mathbb{R}^n)$, $h:\mathbb{R}^n\times \mathbb{R}^m \to \mathbb{R}^q$, $x(t,\omega) \in \mathbb{R}^n$, and $u(t,\omega) \in \mathbb{R}^m$. If the stochastic processes are evolving on a manifold, then this definition can be directly extended by considering a generic form of explicit SDAEs on manifolds as
\begin{subequations}
\label{equation:SDAE on manifold strato}
\begin{equation}
\label{equation:SDAE on manifold - eq strato}
\delta X_t = V(X_t,U_t) dt + \sum_{l = 1}^d\sigma_l(X_t,U_t)\circ dW^l_t,
\end{equation}
\begin{equation}
\label{equation:SDAE on manifold - constraint strato}
h(X_t,U_t) = p;
\end{equation}
\end{subequations}
where the SDE in equation \eqref{equation:SDAE on manifold - eq strato} is in Stratonovich representation, $X_t\in M$, $U_t\in N$, $p\in P$, $h:M\times N \to P$, $V:M\times N\to TM$ such that $\tau_M(V(\cdot, u)) = Id_M$ for all $u\in N$, $\sigma_l:M\times N\to TM$ such that $\tau_M(\sigma_l(\cdot, u)) = Id_M$ for all $u\in N$ and $l\in \{1, 2, ..., d\}$; and $M$, $N$, and $P$ are smooth manifolds of dimension $n$, $m$, and $q$ respectively. However, on manifolds, Stratonovich representation of SDEs make the stochastic analysis in the mean sense difficult. This difficulty is avoided if the SDE is represented in the standard Ito form. In \cite{gliklikh2011global}, readers will find that the standard Ito SDE representation requires the manifold to be equipped with a connection. If the manifolds is not equipped with a connection then the only way to avoid the difficulty in stochastic analysis in mean sense, is by considering the Intrinsic representation, which is the most general way of representing SDEs on manifolds. The Intrinsic form of representing SDEs depends on Schwartz morphism that morphs a semi-martingale from one manifold to the other. A few techniques of constructing Schwartz morphism from $\mathbb{R}^{p+1}\to M$ such that it morphs the stochastic process $(t,W_t)\in \mathbb{R}^{p+1}$ into a semi-martingale on $M$, where $W_t$ is a $p$-dimensional Wiener process, have been demonstrated in \cite{suthar2022Intrinsic}. These constructions rely on a special type of fiber preserving maps from the tangent bundle to the diffusion bundle. These maps are called diffusion generators. Using the diffusion generator approach, it is possible to write the SDE in terms of the differential $dt$ and the Ito differential $dW_t$, hence making the stochastic analysis in mean sense easier. In this article, we combine the ideas on explicitly constrained SDEs on Euclidean spaces from \cite{suthar2021explicit} with that on Intrinsic SDEs via diffusion generator approach from \cite{suthar2022Intrinsic}.

\subsection{Key results.}
The main results of this manuscript can be found in section \ref{section:explicit SDAEs on mani} and section \ref{section:approximate solution}. In section \ref{section:explicit SDAEs on mani} we have introduced the notion of index for SDAEs on manifolds. We have proven the existence of a class of SDAEs that have no solution. We call these equations as \textit{ill-posed} SDAEs. The existence of these type of equations make it difficult to have a general method to solve high index SDAEs. However, if we relax the constraint, we find that the stochastic process may hit the cut-locus where the distance function loses its differentiability. We have discussed such approximate relaxed solution in section \ref{section:approximate solution} and have also shown that it is possible to reduce the probability of hitting the cut-locus. We have presented an algorithm for numerical computations in section \ref{section:numerical}. We have considered an example of SDAE on a sphere to demonstrate the algorithm.

\subsection{Basic definitions and notations.}
We will be following the notational convention from \cite{suthar2022Intrinsic}. Sections of any fiber bundle $F$ are denoted by $\Gamma(F)$ and the set of all smooth vector fields on manifold $M$ by $\mathfrak{X}(M)$. The set of all smooth functions is denoted by $\mathfrak{F}(M)$. We call the {\textit{second order tangent bundle}} of an n-manifold $M$ as \textbf{\textit{diffusion bundle}} to avoid confusing it with second tangent bundle TTM.  It is denoted by $\mathfrak{D}M$. The fiber of the diffusion bundle at point $x\in M$ is called diffusion space and is denoted by $\mathfrak{D}_xM$. The elements of the diffusion space are called \textbf{\textit{diffusors}}. A \textbf{\textit{diffusor field}} is defined as the smooth sections of the diffusion bundle $\mathfrak{D}M$. Following the notational convention for section of a fiber bundle, the set of all smooth diffusor field is denoted as $\Gamma(\mathfrak{D}M)$. If $\phi:M\to N$, the fiber preserving map over $\phi$, $\mathfrak{D}\phi:\mathfrak{D}M\to \mathfrak{D}N$ is called the \textit{\textbf{diffusion map}}. Locally in charts $(U,\Upsilon)$ on $M$ and $(V,\chi)$ on $N$, for all $L\in \mathfrak{D}M$ such that $L|_U = a^i\partial_i + b^{ij}\partial^2_{ij}$,
\[\mathfrak{D}\phi \left(L\right)|_V = \left[ a^i\partial_i\phi^k + b^{ij}\partial^2_{ij}\phi^k\right] \partial_k + \left[b^{ij}\partial_i\phi^k \partial_j\phi^l\right]\partial^2_{kl}.\]
Given $L\in \mathfrak{D}_xM$, there exists a symmetric contravariant tensor $\hat{L}\in T^2_0M$ such that
\[\hat{L}(df, dg) = \dfrac{1}{2}(L[fg] - fL[g] - gL[f]).\] Moreover, if $L|_U = a^i\partial_i + b^{ij}\partial^2_{ij}$ then locally \[\hat{L}(df,dg) = b^{ij}\partial_if\partial_jg.\]
\noindent In other words, $\hat{\cdot} : \mathfrak{D}M\to T_0^2M$ such that $\hat{L}$ is the  symmetric part of the diffusor $L$.

\section{Background.}
In \cite{suthar2021explicit}, a system of explicit Stochastic Differential Algebraic Equations on Euclidean spaces is defined as
\begin{subequations}
\label{eq:SDAE1}
\begin{equation}
\label{eq:SDAE-SDE}
dx = f(x,u)dt + \sigma(x,u)dW_t,  
\end{equation}
\begin{equation}
\label{eq:SDAE-AE}
    h(x,u) = - \int_0^t\Gamma(x(s),u(s))dW_s;
\end{equation}
\end{subequations}
where $W_t = (W^1_t, W^2_t, ..., W^d_t)$ is a d-dimensional Wiener process, $f:\mathbb{R}^n\times \mathbb{R}^m\to \mathbb{R}^n$, and $\sigma:\mathbb{R}^n\times \mathbb{R}^m\to L(\mathbb{R}^d,\mathbb{R}^n)$, $h:\mathbb{R}^n\times \mathbb{R}^m \to \mathbb{R}^q$, $\Gamma:\mathbb{R}^n\times \mathbb{R}^m \to L(\mathbb{R}^d,\mathbb{R}^q)$, $x(t,\omega) \in \mathbb{R}^n$, and $u(t,\omega) \in \mathbb{R}^m$. In the same article it has been demonstrated that it is possible to convert a given SDAE with noisy constraint into an SDAE with noiseless constraint. Hence, for SDAEs on manifolds we can simply consider constraints without noise. A generic form of explicit SDAE on manifolds, in Stratonovich representation, can be given as
\begin{subequations}
\begin{equation}
\delta X_t = V(X_t,U_t) dt + \sum_{l = 1}^d\sigma_l(X_t,U_t)\circ dW^l_t,
\tag{\ref{equation:SDAE on manifold - eq strato}}
\end{equation}
\begin{equation}
h(X_t,U_t) = p;
\tag{\ref{equation:SDAE on manifold - constraint strato}}
\end{equation}
\end{subequations}
where $X_t\in M$, $U_t\in N$, $p\in P$, $h:M\times N \to P$, $V:M\times N\to TM$ such that $\tau_M(V(\cdot, u)) = Id_M$ for all $u\in N$, $\sigma_l:M\times N\to TM$ such that $\tau_M(\sigma_l(\cdot, u)) = Id_M$ for all $u\in N$ and $l\in \{1, 2, ..., d\}$; and $M$, $N$, and $P$ are smooth manifolds of dimension $n$, $m$, and $q$ respectively. However, as discussed in the introduction, on manifolds it is convenient to represent SDEs in Intrinsic form that depends on the Schwartz's second order stochastic differential geometry. To recall, if we consider two manifolds $M$ and $N$ with $x\in M$ and $y\in N$, then if there exists a linear map $J(x,y):\mathfrak{D}_xM\to \mathfrak{D}_yN$ such that $Img(J|_{T_xM})\subset T_yN$ and $\widehat{JL} = (J|_{T_xM}\otimes J|_{T_xM}) \hat{L},$ then such a map $J$ is called a {\textit{Schwartz morphism}}. In Schwartz's second order stochastic differential geometry, the key idea is the infinitesimal representation of a semi-martingale. For a semi-martingale $X_t$ on a manifold $M$, an infinitesimal increment $\mathbf{d}X_t$ is defined as an element of the diffusion space at $X_t\in M$. This infinitesimal increment is called the \textit{Intrinsic differential} or the \textit{Schwartz differential}. Using this differential, and the Schwartz morphism $J$ from $M$ to $N$, it is possible to morph the semi-martingale $X_t\in M$ into a semi-martingale $Y_t\in N$, such that locally we have
\[\mathbf{d}Y_t = J(X_t,Y_t)\mathbf{d}X_t.\]
These equation have been called intrinsic SDEs in \cite{emery2012stochastic}. Reader's can refer \cite{schwartz1982geometrie,emery2012stochastic} for Schwartz's second order stochastic differential geometry.

In \cite{suthar2022Intrinsic}, a techniques of constructing Schwartz morphism from $\mathbb{R}^{p+1}\to M$ have been presented. A Schwartz morphism constructed using these techniques can transform the stochastic process $(t,W_t)\in \mathbb{R}^{p+1}$ into a semi-martingale on $M$, where $W_t$ is a $p$-dimensional Wiener process. The key idea of this construction, presented in \cite{suthar2022Intrinsic}, is that of \textit{diffusion generator}. A fiber preserving map $G:TM\to \mathfrak{D}M$ is called a \textbf{\textit{diffusion generator}} if $\forall$ $Y\in TM$, $\widehat{G(Y)} = Y\otimes Y$. As per the diffusion generator approach, a semi-martingale $X_t\in M$ is given by the Intrinsic SDE,
\begin{equation}
\label{equation:Intrinsic_SDE_diffusion_generator}
\mathbf{d}X_t = \left[V(X_t) + \dfrac{1}{2} \sum_{l = 1}^p G(\sigma_l(X_t))\right]dt + \sum_{l = 1}^p\sigma_l(X_t) dW^l_t,
\end{equation}
where $V\in\mathfrak{X}(M)$ and $\sigma_l\in \mathfrak{X}(M)$ for all $l\in \{1, 2, ..., p\}$.

Using this way of representing the SDEs on manifolds, we can define a generic form of explicit SDAE on manifolds as
\begin{subequations}
\begin{equation*}
\mathbf{d}X_t = \left[V(X_t,U_t) + \dfrac{1}{2} \sum_{l = 1}^d  {}^MG(\sigma_l(X_t,U_t))\right]dt + \sum_{l = 1}^d\sigma_l(X_t,U_t) dW^l_t,
\end{equation*}
\begin{equation*}
h(X_t,U_t) = p;
\end{equation*}
\end{subequations}
where $X_t\in M$, $U_t\in N$, $p\in P$, $h:M\times N \to P$, $V:M\times N\to TM$ such that $\tau_M(V(\cdot, u)) = Id_M$ for all $u\in N$, $\sigma_l:M\times N\to TM$ such that $\tau_M(\sigma_l(\cdot, u)) = Id_M$ for all $u\in N$ and $l\in \{1, 2, ..., d\}$, ${}^MG$ is a diffusion generator on $M$. We are interested in formulating techniques to solve these equations. For this purpose we will convert the given Intrinsic SDE into Stratonovich SDE for which the usual differential geometry can be applied. Once the solution is found, we can convert it back into Intrinsic representation. In \cite{suthar2022Intrinsic}, it has been shown that the equivalent Stratonovich representation of the intrinsic SDE \eqref{equation:Intrinsic_SDE_diffusion_generator} is given by
\begin{equation}
\label{equation:Stratonovich for Intrinsic}
\delta X_t = \left[V(X_t) + \dfrac{1}{2} \sum_{l = 1}^p \nabla^G_S(\sigma_l(X_t))\right]dt + \sum_{l = 1}^p\sigma_l(X_t)\circ dW^l_t,
\end{equation}
where $\nabla^G_S = (G - G_S): TM\to TM$ is a fiber preserving map over identity, and $G_S$ is a diffusion generator called the \textbf{\textit{Stratonovich diffusion generator}}; and the equivalent standard Ito representation is given by
\begin{equation}
\label{equation: standard Ito for Intrinsic}
dX_t = \left[V(X_t) + \dfrac{1}{2} \sum_{l = 1}^p \nabla^G_I(\sigma_l(X_t))\right]dt + \sum_{l = 1}^p\sigma_l(X_t) dW^l_t,
\end{equation}
where $\nabla^G_I = (G - G_I):TM\to TM)$ is a fiber preserving map over identity, and $G_I$ is a diffusion generator called the \textbf{\textit{Ito diffusion generator}}. However, in order to define $G_I$, the manifold $M$ must be equipped with a connection. Readers can refer to \cite{gliklikh2011global} for further details on standard Ito representation of SDEs. In short hand notation we can write these conversion formulae as,
\begin{equation}
\label{equation:conversion_shorthand_stratonovich}
\delta X_t = \mathbf{d}X_t  - \dfrac{1}{2}\sum_{l = 1}^pG_S(\sigma_l(X_t)) dt,
\end{equation}
and
\begin{equation}
\label{equation:conversion_shorthand_Ito}
d X_t = \mathbf{d}X_t  - \dfrac{1}{2}\sum_{l = 1}^pG_I(\sigma_l(X_t)) dt
\end{equation}
respectively.

\section{Explicit Stochastic Differential-Algebraic Equations on manifolds.}
\label{section:explicit SDAEs on mani}
\begin{definition}
\textbf{\textit{Explicit Stochastic Differential-Algebraic Equations (SDAEs) on manifolds}} are defined as the following system of equations
\begin{subequations}
\label{equation:SDAE on manifold}
\begin{equation}
\label{equation:SDAE on manifold - eq}
\mathbf{d}X_t = \left[V(X_t,U_t) + \dfrac{1}{2} \sum_{l = 1}^d  {}^MG(\sigma_l(X_t,U_t))\right]dt + \sum_{l = 1}^d\sigma_l(X_t,U_t) dW^l_t,
\end{equation}
\begin{equation}
\label{equation:SDAE on manifold - constraint}
h(X_t,U_t) = p;
\end{equation}
\end{subequations}
where $X_t\in M$, $U_t\in N$, $p\in P$, $h:M\times N \to P$, $V:M\times N\to TM$ such that $\tau_M(V(\cdot, u)) = Id_M$ for all $u\in N$, $\sigma_l:M\times N\to TM$ such that $\tau_M(\sigma_l(\cdot, u)) = Id_M$ for all $u\in N$ and $l\in \{1, 2, ..., d\}$, ${}^MG$ is a diffusion generator on $M$; and $M$, $N$, and $P$ are smooth manifolds of dimension $n$, $m$, and $q$ respectively. A \textbf{\textit{solution}} of an SDAE is defined as a semi-martingale $(X_t,U_t)\in M\times N$ such that it satisfies both equation \eqref{equation:SDAE on manifold - eq} and equation \eqref{equation:SDAE on manifold - constraint}.
\end{definition}

The usual way of ensuring that the constraint is satisfied is by ensuring that the differential evolution of the constraint is zero. Hence, if we consider the Stratonovich differential of the constraint equation  \eqref{equation:SDAE on manifold - constraint} to be zero, then we obtain
\begin{equation}
D_1h\circ \delta X_t + D_2h \circ\delta U_t = 0.
\end{equation}
Therefore, as long as $D_2h$ is invertible, we get
\begin{equation}
\delta U_t = -(D_2h)^{-1}D_1h \circ \delta X_t.
\end{equation}
Since equation  \eqref{equation:SDAE on manifold - eq} is the Intrinsic SDE for the stochastic process $X_t$, we must convert the Intrinsic differential $\mathbf{d}X_t$ into the Stratonovich differential $\delta X_t$. Therefore, using the conversion formula given by equation  \eqref{equation:conversion_shorthand_stratonovich},
\begin{equation}
\delta U_t = -\left((D_2h)^{-1}D_1h \right) \mathbf{d}X_t  + \dfrac{1}{2}(D_2h)^{-1}D_1h \sum_{l = 1}^d{}^MG_S(\sigma_l(X_t,U_t)) dt.
\end{equation}
Therefore, following the nomenclature from \cite{suthar2021explicit}, we give the following definition.
\begin{definition}
An SDAE in form of equation  \eqref{equation:SDAE on manifold} will be called an \textbf{\textit{index 1 SDAE}}, if $D_2(h)(x,u)$ is invertible. 
\end{definition}
\begin{lemma}
\label{lem:SDEforSDAE}
For an index 1 SDAE in the form of equation  \eqref{equation:SDAE on manifold}, the local solution (local in time) is given by the following Stratonovich SDE.
\begin{equation}
\begin{aligned}
\delta \begin{bmatrix}X_t\\ U_t \end{bmatrix} = \begin{bmatrix}
V(X_t,U_t) + \dfrac{1}{2} \sum_{l = 1}^d  \nabla^{{}^MG}_S (\sigma_l(X_t,U_t))\\ -\left((D_2h)^{-1}D_1h \right) \left(V(X_t,U_t) +\dfrac{1}{2} \sum_{l = 1}^d  \nabla^{{}^MG}_S (\sigma_l(X_t,U_t))\right)\end{bmatrix} dt\\ + \sum_{l = 1}^d \begin{bmatrix}\sigma_l(X_t,U_t)\\ -\left((D_2h)^{-1}D_1h \right)\sigma_l(X_t,U_t) \end{bmatrix}\circ dW^l_t.
\end{aligned}
\end{equation}
\end{lemma}
\noindent The problem starts when $D_2h$ is not invertible. We will call such equations as \textit{high index SDAEs}. A clear case in which $D_2h$ is not invertible is the one in which $h$ is not a function of the algebraic variable. 
\begin{definition}
\label{def:high_index}
An SDAE that does not admit algebraic variable in the algebraic equation (i.e. when $h$ is not a function of the algebraic variable) will be called {\textbf{completely high index}} SDAE. A generic form for high index SDAEs on manifolds is given as,
\begin{subequations}
\label{equation: high index SDAE}
\begin{equation}
\label{equation: high index SDAE - SDE}
\mathbf{d}X_t = \left[V(X_t,U_t) + \dfrac{1}{2} \sum_{l = 1}^d  {}^MG(\sigma_l(X_t,U_t))\right]dt + \sum_{l = 1}^d\sigma_l(X_t,U_t) dW^l_t,
\end{equation}
\begin{equation}
\label{equation: high index SDAE - AE}
p = h(X_t)
\end{equation}
\end{subequations}
where $X_t\in M$, $U_t\in N$, $p\in P$, $h:M \to P$, $V:M\times N\to TM$ such that $\tau_M(V(\cdot, u)) = Id_M$ for all $u\in N$, $\sigma_l:M\times N\to TM$ such that $\tau_M(\sigma_l(\cdot, u)) = Id_M$ for all $u\in N$ and $l\in \{1, 2, ..., d\}$, ${}^MG$ is a diffusion generator on $M$; and $M$, $N$, and $P$ are smooth manifolds of dimension $n$, $m$, and $q$ respectively.
\end{definition}

From this definition it is clear that the solution in lemma \ref{lem:SDEforSDAE} is not valid for high index SDAEs. The following theorem from \cite{suthar2021explicit}, states the necessary condition for existence of a local solution of SDAE in Euclidean spaces.
\begin{theorem}
\label{corollary:subbundle}
Consider a high index SDAE given as 
\begin{subequations}
\label{eq:SDAE2}
\begin{equation}
\label{eq:SDAE2-SDE}
dx = f(x,u)dt + \sum_{l = 1}^d\sigma_l(x,u)dW^l_t,  
\end{equation}
\begin{equation}
\label{eq:SDAE2-AE}
0 = h(x);
\end{equation}
\end{subequations}
with $x(t)\in \mathbb{R}^n$, $u(t)\in \mathbb{R}^m$, $W_t = (W^1_t, W^2_t, ..., W^d_t)$ as d-dimensional Wiener process, $f:\mathbb{R}^n\times \mathbb{R}^m\to \mathbb{R}^n$, $\sigma_l:\mathbb{R}^n\times \mathbb{R}^m\to \mathbb{R}^n$, $h:\mathbb{R}^n\to \mathbb{R}^q$. Let $h(x)$ be $\mathcal{C}^3$, all the vector fields be $\mathcal{C}^1$, and assume that there is a solution of the SDAE. Then
\[ span(\sigma_1 (x(t,\omega),u(t,\omega)), \sigma_2 (x(t,\omega),u(t,\omega)), ..., \sigma_d (x(t,\omega),u(t,\omega)))\subset Ker(D_{x(t,\omega)}h).\]
\end{theorem}
\noindent On manifolds, we will consider a stronger version.
\begin{proposition}
\label{corollary:subbundle_manifold}
Consider a high index SDAE given by equation  \eqref{equation: high index SDAE}. Assume that $(X_t, U_t)\in M\times N$ is the solution for equation  \eqref{equation: high index SDAE}. Then
\[ span(\sigma_1 (X_t,U_t), \sigma_2 (X_t,U_t), ..., \sigma_d (X_t,U_t))\subset Ker(D_{X_t}h),\]
and $V(X_t,U_t) + \dfrac{1}{2} \sum_{l = 1}^d \nabla^{{}^MG}_S(\sigma_l(X_t,U_t))\in Ker(D_{X_t}h)$.
\end{proposition}
\begin{proof}
Since $h(X_t) = p$, \[D_{X_t}h\circ \delta X_t  = 0.\]
Therefore, $\delta X_t \in Ker(D_{X_t}h)$. As
\begin{equation}
\delta X_t = \left[V(X_t,U_t) + \dfrac{1}{2} \sum_{l = 1}^d \nabla^{{}^MG}_S(\sigma_l(X_t,U_t))\right]dt + \sum_{l = 1}^d\sigma_l(X_t,U_t)\circ dW^l_t,
\end{equation}
and $\delta X_t \in Ker(D_{X_t}h)$ for all sample paths,
\[ span(\sigma_1 (X_t,U_t), \sigma_2 (X_t,U_t), ..., \sigma_d (X_t,U_t))\subset Ker(D_{X_t}h),\]
and $V(X_t,U_t) + \dfrac{1}{2} \sum_{l = 1}^d \nabla^{{}^MG}_S(\sigma_l(X_t,U_t))\in Ker(D_{X_t}h)$. Hence, proved.
\end{proof}
\noindent Clearly, if the vector fields $\sigma_l\in \mathfrak{X}(M)$ are not functions of the algebraic variable and if any of the vector fields $\sigma_l(x)\not\in Ker(D_{x}h)$, then there is no semi-martingale $(X_t,U_t)\in M\times N$ that satisfies the SDAE. Following the nomenclature from \cite{suthar2021explicit}, we will call such completely high index SDAEs as \textbf{\textit{ill-posed SDAEs}}.
\begin{corollary_p}
Ill-posed SDAEs have no solution.
\end{corollary_p}
\section{Approximate solution of high index SDAEs.}
\label{section:approximate solution}
In the previous section we have observed that ill-posed SDAEs have no solution. In particular, it is the presence of constraint that makes the SDAE ill-posed. Therefore, it is natural that we relax the constraint and try to find a semi-martingale $(X_t,U_t)$ such that it satisfies the constraint approximately.
\subsection{Approximate solution with unit probability.}
For completely high index SDAEs the constraint equation is given by the function $h:M\to P$ such that for the solution $(X_t,U_t)\in M\times N$, $h(X_t) = p$. For SDAEs that are not solvable, we can relax the constraint by demanding that $\Delta_p(h(X_t))\leq\epsilon$ for some $\epsilon>0$, where $\Delta_p:P\to \mathbb{R}$ is the distance function from point $p$ on the manifold $P$. But in doing so, we are also demanding that the manifold $P$ has a Riemannian structure in order for the distance function to exist.
\begin{definition}
A completely high index SDAE given by equation  \eqref{equation: high index SDAE} will be said to have an $\mathbf{\epsilon}$-\textbf{approximate solution with unit probability} if there exists a semi-martingale $(X_t,U_t)\in M\times N$ such that $\Delta_p(h(X_t))\leq\epsilon$ for some $\epsilon>0$.
\end{definition}
For finding this solution, we will adopt the same strategy that is used in \cite{suthar2021explicit} for SDAE on Euclidean spaces.

If we consider the solution to depend on the algebraic variable, i.e. if $X_t = y(t,U_t)$ where $y:\mathbb{R}\times N \to M$, then the SDE for $X_t$ can be given by the SDE for $y(t,U_t)$
\begin{equation}
\mathbf{d}y(t,U_t) = \left[V(y(t,U_t),U_t) + \dfrac{1}{2} \sum_{l = 1}^d  {}^MG(\sigma_l(y(t,U_t),U_t))\right]dt + \sum_{l = 1}^d\sigma_l(y(t,U_t),U_t) dW^l_t.
\end{equation}
If we assume that $\delta U_t = a(t,U_t) dt + \sum_{l = 1}^d B_l(t,U_t) \circ dW^l_t$, then using the extended Ito formula from \cite{suthar2022Intrinsic}, we can obtain an SDE for $y(t,u)$, for fixed $u\in N$. Let $K :\mathbb{R}\times N\to P$ such that $K(t,u) = h(y(t,u))$. This means that if $y(t,u)$ is such that $\Delta_p(K(t,u))\leq \epsilon$, then $\Delta_p(K(t,U_t))\leq \epsilon$. Furthermore, if $\delta K(t,u) = 0$ and $\Delta_p(K(0,u))\leq \epsilon$ for all $u\in N$, then $\Delta_p(K(t,U_t))\leq \epsilon$ for all $t$ in the interval of existence. Hence, finding an approximate solution boils down to finding the non-autonomous vector fields $a$ and $B_l$ on the manifold $N$ such that $\delta K(t,u) = 0$.

\begin{proposition}
Consider a high index SDAE given by equation  \eqref{equation: high index SDAE}. If there exists a function $Z:N\to P$ such that $Z(u) = h(y(0,u))$ and $\Delta_p(Z(u))\leq \epsilon$ for all $u\in N$, then there exist an $\epsilon$-approximate solution for the SDAE that is given by the solution of the following SDE.
\begin{subequations}
\begin{equation}
\mathbf{d}X_t = \left[V(X_t,U_t) + \dfrac{1}{2} \sum_{l = 1}^d  {}^MG(\sigma_l(X_t,U_t))\right]dt + \sum_{l = 1}^d\sigma_l(X_t,U_t) dW^l_t,
\end{equation}
\begin{equation}
\delta U_t = a(t,U_t) dt + \sum_{l = 1}^d B_l(t,U_t) \circ dW^l_t,
\end{equation}
\end{subequations}
where
\begin{subequations}
\begin{equation}
B_l(t,u) = (Dh(y(t,u))D_2y(t,u))^{-1}Dh(y(t,u))\sigma_l (y(t,u),u),
\end{equation}
\begin{equation}
a(t,u) = (Dh(y(t,u))D_2y(t,u))^{-1}\left[ Dh(y(t,u)) \left( V(y(t,u),u) + \dfrac{1}{2} \sum_{l = 1}^d (\nabla^{{}^M G}_S (\sigma_l (y(t,u),u)) \right)\right],
\end{equation}
\end{subequations}
and $y:\mathbb{R}\times N \to M$ such that $y(t,U_t) = X_t$.
\end{proposition}
\begin{proof}
If we assume that $\delta U_t = a(t,U_t) dt + \sum_{l = 1}^d B_l(t,U_t) \circ dW^l_t$ then using the extended Ito formula from \cite{suthar2022Intrinsic} , if we consider
\begin{multline}
\label{equation:gen_ito_for_X_t}
\delta y(t,u)  = \left[ V(y(t,u),u) - D_2y(t,u)a(t,u) + \dfrac{1}{2} \sum_{l = 1}^d \nabla^{{}^M G}_S(\sigma_l (y(t,u),u)) \right] dt\\ + \sum_{l=1}^d \left[ \sigma_l (y(t,u),u) - D_2y(t,u)B_l(t,u)\right]dW^l_t
\end{multline}
for all $u\in N$, then
\[\mathbf{d}y(t,U_t) = \left[V(y(t,U_t),U_t) + \dfrac{1}{2} \sum_{l = 1}^d  {}^MG(\sigma_l(y(t,U_t),U_t))\right]dt + \sum_{l = 1}^d\sigma_l(y(t,U_t),U_t) dW^l_t.\]
Let $K :\mathbb{R}\times N\to P$ such that $K(t,u) = h(y(t,u))$. This means that if $y(t,u)$ is such that $\Delta_p(K(t,u))\leq \epsilon$, then $\Delta_p(K(t,U_t))\leq \epsilon$.
Now we know that $\delta K(t,u) = Dh(y(t,u))\delta y(t,u)$. Therefore from equation  \eqref{equation:gen_ito_for_X_t},
\begin{multline}
\delta K(t,u) = Dh(y(t,u))\left[ V(y(t,u),u)  - (D_2y(t,u)a(t,u)))\right]dt \\ + Dh(y(t,u))\left[\dfrac{1}{2} \sum_{l = 1}^d (\nabla^{{}^M G}_S (\sigma_l (y(t,u),u)) \right]dt\\ +Dh(y(t,u)) \sum_{l=1}^d \left[ \sigma_l (y(t,u),u) - D_2y(t,u)B_l(t,u)\right]dW^l_t.
\end{multline}
Hence if we want $\delta K(t,u) = 0$, then
\begin{equation}
\label{equation:gen_ito_solution_B_l}
B_l(t,u) = (Dh(y(t,u))D_2y(t,u))^{-1}Dh(y(t,u))\sigma_l (y(t,u),u),
\end{equation}
and
\begin{equation}
\label{equation:gen_ito_solution_a}
a(t,u) = (Dh(y(t,u))D_2y(t,u))^{-1}\left[ Dh(y(t,u)) \left( V(y(t,u),u) + \dfrac{1}{2} \sum_{l = 1}^d (\nabla^{{}^M G}_S (\sigma_l (y(t,u),u)) \right)\right].
\end{equation}
As $\delta K(t,u) = 0$, it implies that if $\Delta_p(K(0,u))\leq \epsilon$ for all $u\in N$, then $\Delta_p(K(t,u))\leq \epsilon$ for all $t$ in the interval of existence and for all $u\in N$. Therefore if $\Delta_p(K(t,u))\leq \epsilon$ for all $u\in N$, then $\Delta_p(K(t,U_t))\leq \epsilon$ for all $t$ in the interval of existence. Finally, if we let $K(0,u)$ equal to the given function $Z(u)$, then as $\Delta_p(Z(u))\leq \epsilon$ is given,  $\Delta_p(K(0,u))\leq \epsilon$ for all $u\in N$. Therefore, we get $\Delta_p(h(X_t))\leq \epsilon$.
\end{proof}

This approach of obtaining approximate solution has been considered in \cite{suthar2021explicit} for SDAEs on Euclidean spaces. However, unlike SDAEs on Euclidean spaces, on manifolds the construction of function $Z:N\to P$ such that $\Delta_p(Z(u)) \leq \epsilon$ for all $u\in N$ is not easy. Nonetheless, if we could construct such a function and equate $Z(u) = h(y(0,u))$, then $\delta U_t = a(t,U_t) dt + \sum_{l = 1}^d B_l(t,U_t) \circ dW^l_t$ will ensure that $\Delta_p(h(X_t))\leq \epsilon$, where the vector fields $a$ and $B_l$ are given by equation \eqref{equation:gen_ito_solution_a} and equation \eqref{equation:gen_ito_solution_B_l} respectively.

\subsection{Approximate solution with bounded probability.}
Based of the method presented in the previous section, the approximate solution with unit probability depends on the construction of the function $Z:N\to P$ such that $\Delta_p(Z(u)) \leq \epsilon$ for all $u\in N$. However, this is not an easy task. For this reason, we find it convenient to further relax the constraint. Instead of allowing the constraint to be violated upto a distance of $\epsilon$, we now seek an approximate solution such that the probability that the constraint is violated beyond $\epsilon$ is bounded. On Euclidean spaces, such type of approximate solution is obtained using the conventional stability analysis for SDEs (check \cite{suthar2021explicit}). On manifolds, however, these techniques can not be adopted directly.

We know that for any function $f\in \mathfrak{F}(M)$ and a semi-martingale $X_t$,
\[f(X_t) - f(X_0)= \int_0^t(\textbf{d}X_s[f])(X_s).\]
If $\mathbf{d}X_t = \left[V(X_t) + \dfrac{1}{2} \sum_{l = 1}^d G(\sigma_l(X_t))\right]dt + \sum_{l = 1}^d\sigma_l(X_t) dW^l_t,$
then 
 \[f(X_t) = f(X_0) + \int_0^t\left[\left(V+ \dfrac{1}{2} \sum_{l = 1}^d G(\sigma_l)\right)[f]\right](X_s) ds + \sum_{l = 1}^d \int_0^t(\sigma_l[f] )(X_s) dW^l_s.\]
Since this is an Ito SDE on a real line, if the initial condition $X_0$ is not random, then
\begin{equation}
\label{equation:expectation of fX_t}
\mathbb{E}(f(X_{t})) = f(X_0) + \mathbb{E}\left[\int_0^{t} \left( V+ \dfrac{1}{2} \sum_{l = 1}^d G(\sigma_l)\right)[f](X_s) ds\right],
\end{equation}
as long as $\mathbb{E}\int_0^t (\sigma_l[f] )(X_s))^2 ds<\infty$ for all $l\in \{1, 2, ..., d\}$

For stability analysis we will assume that $f\in \mathfrak{F}(M)$ is a positive semi-definite function with $f(y) = 0$, $\left[ V[f] + \dfrac{1}{2} \sum_{l = 1}^d G(\sigma_l)[f]\right](x) < 0$ $\forall$ $x\in M\setminus y$, and $\left[ V[f] + \dfrac{1}{2} \sum_{l = 1}^d G(\sigma_l)[f]\right](y) = 0$.

If $f\in \mathfrak{F}(M)$ were geodesically convex in $M$, then using Jensen's inequality we obtain $f(\mathbb{E} (X_{{t}}))\leq \mathbb{E} (f(X_{{t}}))$ $\forall$ $t>0$. Therefore,  $\mathbb{E}(X_t)\to y$ as $t\to \infty$, i.e. $X_t$ is asymptotically stable in the mean. Unfortunately, globally convex smooth functions may not exist e.g. in case of compact manifolds. Hence, global mean asymptotic stability analysis on manifolds is difficult. On the other hand if we have a locally convex function, say convex in $U\subset M$, then the probability that $X_t\not \in U$ for some $t>0$ is non-zero. Once the process $X_t$ is outside the set $U$, we lose the convexity of the function $f$ and hence, nothing can be concluded about the mean asymptotic stability of the process $X_t$. In case $M$ is a Riemannian manifold, if we considering $f$ as the distance function from point $y\in M$, i.e. $f = \Delta_y:M\to \mathbb{R}$, then the function will be geodesically convex upto the cut-locus of the point $y$. However, the function loses its differentiability on the cut-locus. In deterministic systems, the problem of avoidance of cut-locus has been considered in \cite{simha2018global} in the framework of control of ordinary differential equations for command tracking problems. However, in case of stochastic systems, due to randomness the process $X_t$ may hit the cut-locus of the point $y$. Since the problem of SDAE is a more general form of a control problem (where algebraic variable is similar to the control input), if we could ensure that the probability of being in neighbourhood of the cut-locus is very low, then it is reasonable to consider the distance function from point $y$ as the candidate function.
\begin{definition}
Let $\mathfrak{C}_y\subset M$ denote the cut-locus of point $y$ on a Riemannian manifold $(M,g)$. If there exists a function $f:M\to \mathbb{R}$ such that $f$ is geodesically convex on $M\setminus \mathfrak{C}_y$ then we will call the function $f$ as \textit{\textbf{approximately convex}} for $y$.
\end{definition}

Consider the high index SDAE given by equation \eqref{equation: high index SDAE},
\begin{subequations}
\begin{equation*}
\mathbf{d}X_t = \left[V(X_t,U_t) + \dfrac{1}{2} \sum_{l = 1}^d  {}^MG(\sigma_l(X_t,U_t))\right]dt + \sum_{l = 1}^d\sigma_l(X_t,U_t) dW^l_t,
\end{equation*}
\begin{equation*}
p = h(X_t);
\end{equation*}
\end{subequations}
where $X_t\in M$, $U_t\in N$, $p\in P$, $h:M \to P$, $V:M\times N\to TM$ such that $\tau_M(V(\cdot, u)) = Id_M$ for all $u\in N$, $\sigma_l:M\times N\to TM$ such that $\tau_M(\sigma_l(\cdot, u)) = Id_M$ for all $u\in N$ and $l\in \{1, 2, ..., d\}$, ${}^MG$ is a diffusion generator on $M$; and $M$, $N$, and $P$ are smooth manifolds of dimension $n$, $m$, and $q$ respectively. We will assume that the manifold $P$ is equipped with a Riemannian metric $g$.

Following the definitions introduced in \cite{suthar2021explicit}, for SDAEs on manifolds, we give the following definitions.
\begin{definition}
A completely high index SDAE will be said to have an \textbf{{m-solution}} if there exists semi-martingale $(X_t,U_t) \in M\times N$ such that $\mathbb{E}(h(X_t)) = p$.
\end{definition}

\begin{proposition}
\label{proposition: m-solution}
Let the completely high index SDAE be given by equation \eqref{equation: high index SDAE} with $P$ as the Riemannian manifold. If there exists a positive semi-definite, and approximately convex function $f:P\to \mathbb{R}$ such that $f|_{P\setminus \overline{\mathfrak{C}^\epsilon_p}}\in \mathfrak{F}^3(P\setminus \overline{\mathfrak{C}^\epsilon_p})$ for $p\in P$ and a semi-martingale $(X_t,U_t)\in M\times N$ such that 
\begin{equation}\label{equation:cond_m-sol_mani}
\left(D_{X_t}hV(X_t,U_t)+ \dfrac{1}{2} \sum_{l = 1}^d D_{X_t}h\nabla^{{}^MG}_S(\sigma_l)(X_t,U_t) + \dfrac{1}{2} \sum_{l = 1}^d{}^PG_S(D_{X_t}h\sigma_l)(X_t,U_t) \right)[f(h(X_t))] = 0,
\end{equation}
\[f(p) = 0,\]
\[h(X_0) = p,\]
and $\lambda =\sup_{(x,u)\in U\times V} \sum_{l = 1}^d\left( (D_{x}h)\sigma_l(x,u)[f(h(x))]\right)^2<\infty$ for $U\times V = \{(X_t(\omega),U_t(\omega))\in M\times N| t\in (0,a], \omega\in \Omega\}$, where $\Omega$ is the sample space;
then for any  $t\in (0,a]$, for some $a>0$ that makes $(0,a]$ an interval of existence, then the equation  \eqref{equation: high index SDAE} has a local m-solution. $\overline{\mathfrak{C}^\epsilon_p}$ is the closure of the $\epsilon$ - neighbourhood of cut-locus of $p$, $\mathfrak{C}^\epsilon_p$.
\end{proposition}
\begin{proof}
As \[\delta X_t = \left[ V(X_t,U_t) + \dfrac{1}{2} \sum_{l = 1}^d  \nabla^{{}^MG}_S (\sigma_l(X_t,U_t))\right]dt + \sum_{l = 1}^d\sigma_l(X_t,U_t) dW^l_t,\]
\[\delta h(X_t) = \left[ D_{X_t}hV(X_t,U_t) + \dfrac{1}{2} \sum_{l = 1}^d D_{X_t}h \nabla^{{}^MG}_S (\sigma_l(X_t,U_t))\right]dt + \sum_{l = 1}^dD_{X_t}h\sigma_l(X_t,U_t) dW^l_t.\]
 \[\therefore\mathbf{d}h(X_t) = \left[ D_{X_t}hV(X_t,U_t) + \dfrac{1}{2} \sum_{l = 1}^d D_{X_t}h \nabla^{{}^MG}_S (\sigma_l(X_t,U_t)) + \dfrac{1}{2} \sum_{l = 1}^d{}^PG_S(D_{X_t}h\sigma_l) (X_t,U_t)\right]dt\]
 \[+ \sum_{l = 1}^dD_{X_t}h\sigma_l(X_t,U_t) dW^l_t.\]
From equation  \eqref{equation:expectation of fX_t}, we know that if $h(X_0) = p$, $f(p) = 0$, and 
\[\mathbb{E}\int_0^t (D_{X_s}h\sigma_l(X_s,U_s)[f(h(X_s))])^2 ds <\infty\] for all $l\in \{1, 2, ..., d\}$; then
 \[\mathbb{E}f(h(X_{t})) = \mathbb{E}\int_0^{t}\left[  D_{X_s}hV(X_s,U_s) + \dfrac{1}{2} \sum_{l = 1}^d D_{X_s}h \nabla^{{}^MG}_S (\sigma_l(X_s,U_s))\right][f(h(X_s))] ds\]
 \[+ \mathbb{E}\int_0^{t}\left[\dfrac{1}{2} \sum_{l = 1}^d{}^PG_S(D_{X_s}h\sigma_l) (X_s,U_s)\right][f(h(X_s))] ds\]
 Therefore using equation \eqref{equation:cond_m-sol_mani}, the integrals are zero. Hence,
 \[\mathbb{E}f(h(X_{t})) = 0.\]
The $\mathcal{C}^3$ smoothness of $f:P\to \mathbb{R}$ in $P\setminus \overline{\mathfrak{C}^\epsilon_p}$ is required to ensure that the drift term $D_xhV(x,u) + \dfrac{1}{2} \sum_{l = 1}^d D_xh \nabla^{{}^MG}_S (\sigma_l(x,u)) + \dfrac{1}{2} \sum_{l = 1}^d{}^PG_S(D_{x}h\sigma_l) (x,u)$ is Lipschitz.
Using Jensen's inequality, $f(\mathbb{E}h(X_{t})) = 0$. Hence, $\mathbb{E}(h(X_{t})) = p$.
\end{proof}

\begin{corollary_p}
Let the completely high index SDAE be given by equation  \eqref{equation: high index SDAE}, with $P$ as a Riemannian manifold.  If there exists a semi-martingale $(X_t,U_t)\in M\times N$ such that \[\left(D_{X_t}hV(X_t,U_t)+ \dfrac{1}{2} \sum_{l = 1}^d \left(D_{X_t}h\nabla^{{}^MG}_S(\sigma_l)(X_t,U_t) + {}^PG_S(D_{X_t}h\sigma_l)(X_t,U_t)\right) \right)[\Delta^2_p(h(X_t))] = 0,\]
and
\[h(X_0) = p;\]
then the equation  \eqref{equation: high index SDAE} has an m-solution. $\Delta_p(z)$ is the distance of point $z\in P$ from point $p\in P$.  
\end{corollary_p}
\begin{proof}
$\Delta^2_p(h(X_t))$ is clearly a positive semi-definite function with $\Delta^2_p(p) = 0$. The function is also approximately convex  for $p\in P$. The function $\Delta_p$ is smooth everywhere except on cutlocus and the point $p$. The local smoothness of $\Delta^2_p$ around $p\in P$ follows from the smoothness in the normal coordinates around $p\in P$. Hence, this statement is a direct application of proposition \ref{proposition: m-solution}.
\end{proof}

As discussed before, in order to ensure that the solution has low probability of hitting the cut-locus $\mathfrak{C}_p$, we must have a probability bound on the explosion of the solution.
\begin{definition}
We define an $\epsilon$ \textbf{\textit{{bounded m-solution}}}, with probability of explosion less than or equal to $\alpha$, as an m-solution $(X_t,U_t)\in M\times N$ such that for all $t>0$, 
\[\mathbb{P}\left(\sup_{s\in [0,t]}(\Delta_p(h(X_s))>\epsilon\right)\leq \alpha\]
for some $\alpha\in (0,1]$ and $\epsilon >0$, where $\Delta_p:P\to \mathbb{R}$ such that $\Delta_p(z)$ is the distance of point $z\in P$ from the point $p\in P$. We will call $X_t$ as \textbf{\textit{local bounded m-solution}} if for any $t\in (0,a]$ for some $a>0$ that makes $(0,a]$ an interval of existence,
\[\mathbb{P}\left(\sup_{s\in [0,t]}(\Delta_p(h(X_s))>\epsilon\right)\leq \alpha.\]
\end{definition}

\begin{proposition}
\label{proposition:bounded m-solution}
Let the completely high index SDAE be given by equation  \eqref{equation: high index SDAE},
\begin{subequations}
\begin{equation*}
\mathbf{d}X_t = \left[V(X_t,U_t) + \dfrac{1}{2} \sum_{l = 1}^d  {}^MG(\sigma_l(X_t,U_t))\right]dt + \sum_{l = 1}^d\sigma_l(X_t,U_t) dW^l_t,
\tag{\ref{equation: high index SDAE - SDE}}
\end{equation*}
\begin{equation*}
p = h(X_t),
\tag{\ref{equation: high index SDAE - AE}}
\end{equation*}
\end{subequations}
where $X_t\in M$, $U_t\in N$, $p\in P$, $h:M \to P$, $V:M\times N\to TM$ such that $\tau_M(V(\cdot, u)) = Id_M$ for all $u\in N$, $\sigma_l:M\times N\to TM$ such that $\tau_M(\sigma_l(\cdot, u)) = Id_M$ for all $u\in N$ and $l\in \{1, 2, ..., d\}$, ${}^MG$ is a diffusion generator on $M$; and $M$, $N$, and $P$ are smooth manifolds of dimension $n$, $m$, and $q$ respectively. Let $P$ be a Riemannian manifold. Let $\lambda =\sup_{(x,u)\in U\times V} \sum_{l = 1}^d\left( (D_{x}h)\sigma_l(x,u)[\Delta^2_p(h(x))]\right)^2<\infty$ for $U\times V = \{(X_t(\omega),U_t(\omega))\in M\times N| t\in (0,a], \omega\in \Omega\}$, where $\Omega$ is the sample space.

For every $\epsilon>0$, and $\alpha\in (0,1]$; there exists $b > \lambda/(\alpha {\epsilon^4})$ such that if there exists a semi-martingale $(X_t,U_t)\in M\times N$ that ensures $h(X_0) = p$, and
\begin{multline}
\left(D_{X_t}hV(X_t,U_t)+ \dfrac{1}{2} \sum_{l = 1}^d D_{X_t}h\nabla^{{}^MG}_S(\sigma_l)(X_t,U_t) + \dfrac{1}{2} \sum_{l = 1}^d{}^PG_S(D_{X_t}h\sigma_l)(X_t,U_t) \right)[\Delta^2_p(h(X_t))] \\=- b \Delta^2_p(h(X_t));
\end{multline}
then for any  $t\in (0,a]$, for some $a>0$ that makes $(0,a]$ an interval of existence,
\[\mathbb{P}\left(\sup_{s\in [0,t]}(\Delta_p(h(X_s))>\epsilon\right)\leq \alpha.\]
In other words, the equation  \eqref{equation: high index SDAE} has a local bounded m-solution for $t\in (0,a]$.
\end{proposition}
\begin{proof}
From the previous proposition we know that
\begin{multline*}
f(h(X_{t})) =\\ \int_0^{t}\left[ D_xhV(X_t,U_t) + \dfrac{1}{2} \sum_{l = 1}^d\left( D_xh \nabla^{{}^MG}_S (\sigma_l(X_t,U_t)) + {}^PG_S(D_{X_t}h\sigma_l) (X_t,U_t)\right)\right][f(h(X_s))] ds\\+\int_0^t\sum_{l = 1}^d(D_{X_t}h)\sigma_l(X_s,U_s)[f(h(X_s))] dW^l_s.
\end{multline*}
Therefore, for $f = \Delta^2_p$, if
\begin{multline*}
\left(D_{X_t}hV(X_t,U_t)+ \dfrac{1}{2} \sum_{l = 1}^d\left( D_{X_t}h\nabla^{{}^MG}_S(\sigma_l)(X_t,U_t) + {}^PG_S(D_{X_t}h\sigma_l)(X_t,U_t)\right) \right)[\Delta^2_p(h(X_t))] \\=- b \Delta^2_p(h(X_t)),
\end{multline*}
then 
\[d\Delta^2_p(h(X_t)) = -b \Delta^2_p(h(X_t)) dt + \sum_{l = 1}^d(D_{X_t}h)\sigma_l(X_t,U_t)[\Delta^2_p(h(X_t))] dW^l_t.\]
We know that as $h(X_0) = p$,
\[\Delta^2_p(h(X_t)) = e^{-bt}\int_0^te^{bs}\sum_{l = 1}^d(D_{X_s}h)\sigma_l(X_s,U_s)[\Delta^2_p(h(X_s))] dW^l_s.\]
Furthermore, if $\lambda =\sup_{(x,u)\in U\times V} \sum_{l = 1}^d\left( (D_{x}h)\sigma_l(x,u)[\Delta^2_p(h(x))]\right)^2<\infty$, then
\[{\mathbb{E}\int_0^te^{2bs}\sum_{l = 1}^d\left((D_{X_s}h)\sigma_l(X_s,U_s)[\Delta^2_p(h(X_s))]\right)^2 ds}\leq \lambda(e^{2bt}-1)/2b<\infty \quad \forall\quad t\in (0,a].\]
This makes the integral $\int_0^te^{bs}\sum_{l = 1}^d(D_{X_s}h)\sigma_l(X_s,U_s)[\Delta^2_p(h(X_s))] dW^l_s$ a martingale. Hence, using the Doob's martingale inequality we get,
\[\mathbb{P}\left(\sup_{s\in [0,t]}(e^{bs}\Delta^2_p(h(X_s))>k)\right) \leq \dfrac{\mathbb{E}(e^{2bt}\Delta^4_p(h(X_t)))}{k^2}.\]
\[\therefore \mathbb{P}\left(\sup_{s\in [0,t]}(e^{bs}\Delta^2_p(h(X_s))>k)\right) =  \mathbb{P}\left(\sup_{s\in [0,t]}(\Delta^2_p(h(X_s))>\epsilon)\right) \leq \dfrac{\mathbb{E}(e^{2bt}\Delta^4_p(h(X_t)))}{e^{2bt}\epsilon^2},\]
where $k = e^{bt}\epsilon$.
\[\therefore \mathbb{P}\left(\sup_{s\in [0,t]}(\Delta^2_p(h(X_s))>\epsilon)\right) \leq \dfrac{\mathbb{E}\left(\int_0^te^{bs}\sum_{l = 1}^d(D_{X_s}h)\sigma_l(X_s,U_s)[\Delta^2_p(h(X_s))] dW^l_s\right)^2}{e^{2bt}\epsilon^2}.\]
Using Ito-isometry property,
\[\mathbb{P}\left(\sup_{s\in [0,t]}(\Delta^2_p(h(X_s))>\epsilon)\right) \leq \dfrac{\mathbb{E}\int_0^te^{2bs}\sum_{l = 1}^d\left((D_{X_s}h)\sigma_l(X_s,U_s)[\Delta^2_p(h(X_s))]\right)^2 ds}{e^{2bt}\epsilon^2}<\lambda /2b\epsilon^2.\]

Since this is true for any $0<t\leq a$; if $b > \lambda/(\epsilon^2 \alpha)$, then $\dfrac{\lambda}{2b\epsilon^2}<\alpha$; and hence
\[\mathbb{P}\left(\sup_{s\in [0,t]}(\Delta^2_p(h(X_s))>\epsilon)\right)\leq \alpha.\]
In other words if $b > \lambda/(\alpha \epsilon^4)$, then for any $0<t\leq a$
\[\mathbb{P}\left(\sup_{s\in [0,t]}(\Delta_p(h(X_s))>\epsilon)\right)\leq \alpha.\]
\end{proof}
This proposition proves that for large enough $b>0$, we can ensure that a semi-martingale $(X_t,U_t)\in M\times N$ that satisfies
\begin{subequations}
\begin{equation}
\mathbf{d}X_t = \left[V(X_t,U_t) + \dfrac{1}{2} \sum_{l = 1}^d  {}^MG(\sigma_l(X_t,U_t))\right] dt + \sum_{l = 1}^d\sigma_l(X_t,U_t) dW^l_t,
\end{equation}
\begin{multline}
\left(D_{X_t}hV(X_t,U_t)+ \dfrac{1}{2} \sum_{l = 1}^d D_{X_t}h\nabla^{{}^MG}_S(\sigma_l)(X_t,U_t) + \dfrac{1}{2} \sum_{l = 1}^d{}^PG_S(D_{X_t}h\sigma_l)(X_t,U_t) \right)[\Delta^2_p(h(X_t))]\\ =- b \Delta^2_p(h(X_t)),
\end{multline}
\end{subequations}
is a local bounded m-solution for the completely high index SDAE  \eqref{equation: high index SDAE}. However, the existence of such semi-martingale $(X_t,U_t)\in M\times N$ is something that is not guaranteed by this statement. Numerical computation of this semi-martingale is considered in the next section.
\begin{corollary_p}
\label{corollary:cutlocus_issue}
Let $\mathfrak{I}_p$ be the injectivity radius of point $p\in P$ and let $\alpha\in (0,1]$. If there exists $\epsilon>0$ such that $\epsilon <\mathfrak{I}_p$, and if there exists a semi-martingale $(X_t,U_t)\in M\times N$ such that it satisfies
\begin{subequations}
\begin{equation}
\mathbf{d}X_t = \left[V(X_t,U_t) + \dfrac{1}{2} \sum_{l = 1}^d  {}^MG(\sigma_l(X_t,U_t))\right]dt + \sum_{l = 1}^d\sigma_l(X_t,U_t) dW^l_t \; and
\end{equation}
\begin{multline}
\left(D_{X_t}hV(X_t,U_t)+ \dfrac{1}{2} \sum_{l = 1}^d D_{X_t}h\nabla^{{}^MG}_S(\sigma_l)(X_t,U_t) + \dfrac{1}{2} \sum_{l = 1}^d{}^PG_S(D_{X_t}h\sigma_l)(X_t,U_t) \right)[\Delta^2_p(h(X_t))]\\ =- b \Delta^2_p(h(X_t))
\end{multline}
\end{subequations}
for large enough $b>0$, then $(X_t,U_t)\in M\times N$ is a bounded m-solution for the completely high index SDAE  \eqref{equation: high index SDAE} such that $\mathbb{P}(h(X_t)\in \mathfrak{C}_p\; for\; some\; t\in(0,a])< \alpha$; where $(0,a]$ is the interval of existence.
\end{corollary_p}
\begin{proof}
$\mathfrak{I}_p$ is the injectivity radius from point $p\in P$. As $\mathfrak{C}_p\subset P\setminus\{z| \Delta_p(z)<\mathfrak{I}_p\}$,
\[\mathbb{P}(h(X_t)\in \mathfrak{C}_p\; for\; some\; t\in(0,a])<\mathbb{P}(\sup_{t\in(0,a]}\Delta_p(h(X_t)) > \mathfrak{I}_p).\]
As $\mathbb{P}((\sup_{t\in(0,a]}\Delta_p(h(X_t)))>\mathfrak{I}_p)\leq \mathbb{P}((\sup_{t\in(0,a]}\Delta_p(h(X_t)))>\epsilon)$ for all $\epsilon < \mathfrak{I}_p$, if \[\mathbb{P}((\sup_{t\in(0,a]}\Delta_p(h(X_t)))>\epsilon)\leq \alpha\] and $\epsilon<\mathfrak{I}_p$, then $\mathbb{P}(h(X_t)\in \mathfrak{C}_p\; for\; some\; t\in(0,a])< \alpha$ is true.
\end{proof}
In proposition \ref{proposition:bounded m-solution} we have shown that if there exists a semi-martingale $(X_t,U_t)\in M\times N$ that ensures $h(X_0) = p$, and 
\begin{multline}
\left(D_{X_t}hV(X_t,U_t)+ \dfrac{1}{2} \sum_{l = 1}^d D_{X_t}h\nabla^{{}^MG}_S(\sigma_l)(X_t,U_t) + \dfrac{1}{2} \sum_{l = 1}^d{}^PG_S(D_{X_t}h\sigma_l)(X_t,U_t) \right)[\Delta^2_p(h(X_t))]\\ =- b \Delta^2_p(h(X_t));
\end{multline}
then for large enough $b>0$, \[\mathbb{P}((\sup_{t\in (0,a]}\Delta_p(h(X_t)))>\epsilon)\leq \alpha.\]
The choice of $0<b\in\mathbb{R}$ in turn depends on the interval $(0,a]$.
In real world applications, in order to avoid the problem of hitting the cut-locus, we will choose $\alpha$ to be small enough (say $\alpha = 0.02$) and $\epsilon$ to be much smaller than $\mathfrak{I}_p$. Then from corollary \ref{corollary:cutlocus_issue}, we know that the probability of hitting the cut-locus of the point $p$ will be very low.

\section{Method for numerical computation of the approximate solution with bounded probability.}
\label{section:numerical}
From the results in the previous section on bounded m-solution, we observe that if we define $Y:M\times N\to \mathbb{R}$ such that
\begin{multline}
Y(x,u) = b \Delta^2_p(h(x))\\ + \left(D_{x}hV(x,u)+ \dfrac{1}{2} \sum_{l = 1}^d D_{x}h\nabla^{{}^MG}_S(\sigma_l)(x,u) + \dfrac{1}{2} \sum_{l = 1}^d{}^PG_S(D_{x}h\sigma_l)(x,u) \right)[\Delta^2_p(h(x))],
\end{multline}
then a process $(X_t,U_t)\in M\times N$ gives a bounded m-solution if it satisfies $Y(X_t,U_t) = 0$. This is stated in proposition \ref{proposition:bounded m-solution}, but without explicitly introducing the function $Y$. However, since $dim(N) = m\geq dim(\mathbb{R}) = 1$, implicit function theorem cannot be used to find $U_t$ unless $m = 1$.  This makes the computation of the process $U_t$ difficult. If we assume that $N$ is a Riemannian manifold with $g$ as the metric then we can use gradient descent to minimize $K = Y^2(X_t,\cdot):N\to \mathbb{R}$ i.e. 
the solution of $\dfrac{du}{ds} = -g^\sharp dK$ as $s\to \infty$ at every time $t\in (0,a]$. Of course, the gradient descent has a problem of converging to a local minimum. Therefore, if we start with $U_0\in N$ that ensures $Y(X_0,U_0) = 0$ and consider the initial condition for  $\dot{u} = -g^\sharp dK$ as the solution of the previous time step, then we may be able to avoid local minima. From a computational perspective, implementation of gradient descent at every time step is not recommended as it will unavoidably add to the computational complexity. Moreover, as this solution is not a result of application of implicit function theorem, we do not know if the resulting process $U_t$ will be a semi-martingale. In the following proposition we address this dimensionality issue.
\begin{proposition}
\label{proposition:high_index_approximate_solution}
Let the completely high index SDAE be given by equation  \eqref{equation: high index SDAE},
\begin{subequations}
\begin{equation}
\mathbf{d}X_t = \left[V(X_t,U_t) + \dfrac{1}{2} \sum_{l = 1}^d  {}^MG(\sigma_l(X_t,U_t))\right]dt + \sum_{l = 1}^d\sigma_l(X_t,U_t) dW^l_t,
\tag{\ref{equation: high index SDAE - SDE}}
\end{equation}
\begin{equation}
p = h(X_t), \tag{\ref{equation: high index SDAE - AE}}
\end{equation}
\end{subequations}
where $X_t\in M$, $U_t\in N$ with $N$ as a Riemannian manifold with metric $g$, $p\in P$, $h:M \to P$, $V:M\times N\to TM$ such that $\tau_M(V(\cdot, u)) = Id_M$ for all $u\in N$, $\sigma_l:M\times N\to TM$ such that $\tau_M(\sigma_l(\cdot, u)) = Id_M$ for all $u\in N$ and $l\in \{1, 2, ..., d\}$, ${}^MG$ is a diffusion generator on $M$; and $M$, $N$, and $P$ are smooth manifolds of dimension $n$, $m$, and $q$ respectively. Let $P$ be a Riemannian manifold.  Suppose there exists a semi-martingale $(X_t,U_t)\in M\times N$ that satisfies
\begin{subequations}
\label{equation:solution_bounded-m}
\begin{equation}
\mathbf{d}X_t = \left[V(X_t,U_t) + \dfrac{1}{2} \sum_{l = 1}^d  {}^MG(\sigma_l(X_t,U_t))\right]dt + \sum_{l = 1}^d\sigma_l(X_t,U_t) dW^l_t,
\end{equation}
\begin{equation}
\delta U_t = \alpha_0(X_t,U_t) a(X_t,U_t) dt + \sum_{l = 1}^d \alpha_l(X_t,U_t)a(X_t,U_t) dW^l_t;
\end{equation}
\end{subequations}
where
\[\alpha_0(X_t,U_t) = \dfrac{-D_1Y \left[V(X_t,U_t) + \dfrac{1}{2} \sum_{l = 1}^d  \nabla^{{}^MG}_S(\sigma_l(X_t,U_t))\right]}{D_2Y a(X_t,U_t)},\]
\[\alpha_l(X_t,U_t) = \dfrac{-D_1Y\sigma_l(X_t,U_t)}{D_2Y a(X_t,U_t)},\]
\[a\not \in Ker(D_2Y),\]
$a:M\times N\to TN$ such that $a(x,u)\in T_uN$; and $Y:M\times N\to \mathbb{R}$ such that
\begin{multline*}
Y(x,u) = \\ b \Delta^2_p(h(x)) + \left(D_{x}hV(x,u)+ \dfrac{1}{2} \sum_{l = 1}^d D_{x}h\nabla^{{}^MG}_S(\sigma_l)(x,u) + \dfrac{1}{2} \sum_{l = 1}^d{}^PG_S(D_{x}h\sigma_l)(x,u) \right)[\Delta^2_p(h(x))],
\end{multline*}
where $b > \lambda/(\alpha {\epsilon^4})$. If there exists $U_0\in N$ such that $Y(X_0,U_0) = 0$ and if $D_2Y(X_t,U_t)\not = 0$, then $(X_t,U_t)$ is a bounded m-solution for the high index SDAE  \eqref{equation: high index SDAE} such that 
\[\mathbb{P}((\sup_{t\in(0,a]}\Delta_p(h(X_t)))> \epsilon)< \alpha\] and $\epsilon$ is much smaller than the injectivity radius $\mathfrak{I}_p$ from $p\in P$.
\end{proposition}
\begin{proof}
As $Y:M\times N\to \mathbb{R}$ such that
\begin{multline*}
Y(x,u) =\\ b \Delta^2_p(h(x)) + \left(D_{x}hV(x,u)+ \dfrac{1}{2} \sum_{l = 1}^d D_{x}h\nabla^{{}^MG}_S(\sigma_l)(x,u) + \dfrac{1}{2} \sum_{l = 1}^d{}^PG_S(D_{x}h\sigma_l)(x,u) \right)[\Delta^2_p(h(x))],
\end{multline*}
from proposition \ref{proposition:bounded m-solution}, we know that if there exists a semi-martingale $(X_t,U_t)\in M\times N$ such that it satisfies $Y(X_t,U_t) = 0$, then $\mathbb{P}((\sup_{t\in(0,a]}\Delta_p(h(X_t)))> \epsilon)< \alpha$ for large enough $b>0$.
If we assume that
\[\delta U_t = \alpha_0(X_t,U_t) a(X_t,U_t) dt + \sum_{l = 1}^d \alpha_l(X_t,U_t)a(X_t,U_t) dW^l_t,\]
for some $\alpha_i(X_t,U_t) \in \mathbb{R}$ for $i\in \{0, 1, ..., d\}$, and $a(x,\cdot)\in \mathfrak{X}(N)$; then Stratonovich differentiation of $Y(X_t,U_t)$ gives 
\[D_1Y \left[V(X_t,U_t) + \dfrac{1}{2} \sum_{l = 1}^d  \nabla^{{}^MG}_S(\sigma_l(X_t,U_t))\right]dt + \sum_{l = 1}^d D_1Y\sigma_l(X_t,U_t) dW^l_t\]
\[+ \alpha_0(X_t,U_t) D_2Y a(X_t,U_t) dt +  \sum_{l = 1}^d \alpha_l(X_t,U_t)D_2Y a(X_t,U_t) dW^l_t = 0.\]
If $a\not \in Ker(D_2Y)$, then
\[\alpha_0(X_t,U_t) = \dfrac{-D_1Y \left[V(X_t,U_t) + \dfrac{1}{2} \sum_{l = 1}^d  \nabla^{{}^MG}_S(\sigma_l(X_t,U_t))\right]}{D_2Y a(X_t,U_t)},\]
and
\[\alpha_l(X_t,U_t) = \dfrac{-D_1Y\sigma_l(X_t,U_t)}{D_2Y a(X_t,U_t)}\]
will ensure that $Y(X_t,U_t) = 0$ for any $U_0$ that satisfies $Y(X_0,U_0) = 0$.
\end{proof}
\subsection{Algorithm}
The computation of the initial condition $U_0$ is still a problem. If $N$ is a Riemannian manifold then we can use gradient descent to compute $U_0$. In any case, we will assume that we know the initial condition $U_0$ that ensures $Y(X_0,U_0) = 0$. Now, in proposition \ref{proposition:high_index_approximate_solution}, we have presented a method of computing the bounded m-solution by simply solving an SDE. However, this method depends on choosing a vector field $a(x,\cdot )\in \mathfrak{X}(N)$ such that $a\not\in Ker(D_2Y)$ and choosing $b>\dfrac{\lambda}{\alpha{\epsilon^4}}$. Therefore, the problem is two folds:
\begin{itemize}
\item Computing \[\lambda =\sup_{(x,u)\in U\times V} \sum_{l = 1}^d\left( (D_{x}h)\sigma_l(x,u)[\Delta^2_p(h(x))]\right)^2\] for \[U\times V = \{(X_t(\omega),U_t(\omega))\in M\times N| t\in (0,a], \omega\in \Omega\},\] where $\Omega$ is the sample space. This is necessary to choose $b>\dfrac{\lambda}{\alpha{\epsilon^4}}$ as per proposition \ref{proposition:bounded m-solution}.
\item Choosing the vector field $a(x,\cdot)\in \mathfrak{X}(N)$ such that $a\not\in Ker(D_2Y)$. 
\end{itemize}
The computation of $\lambda =\sup_{(x,u)\in U\times V} \sum_{l = 1}^d\left( (D_{x}h)\sigma_l(x,u)[\Delta^2_p(h(x))]\right)^2$ can be done using Monte-Carlo approach. However, since we are only interested in choosing large enough $b$ that ensures \[\mathbb{P}\left(\sup_{s\in [0,t]}(\Delta_p(h(X_s))>\epsilon\right)\leq \alpha,\]
algorithmically, we can double the current value of $b$ whenever $\Delta_p(h(X_t))>\epsilon$. In this way, we avoid the computation of $\lambda$.
Now for choosing the vector field $a\not\in Ker(D_2Y)$. Let $a(x,u) = g^\sharp_u D_2Y(x,u)$. Therefore, as long as $D_2Y\not =0$, we get the required solution by solving the following SDE:
\begin{subequations}
\label{equation:numerical_bounded-m-sol}
\begin{equation}
\label{equation:numerical_bounded-m-sol:SDEfor X}
\mathbf{d}X_t = \left[V(X_t,U_t) + \dfrac{1}{2} \sum_{l = 1}^d  {}^MG(\sigma_l(X_t,U_t))\right]dt + \sum_{l = 1}^d\sigma_l(X_t,U_t) dW^l_t,
\end{equation}
\begin{multline}
\label{equation:numerical_bounded-m-sol:SDEfor U}
\delta U_t = -D_1Y \left[V(X_t,U_t) + \dfrac{1}{2} \sum_{l = 1}^d  \nabla^{{}^MG}_S(\sigma_l(X_t,U_t))\right]\dfrac{g^\sharp_{U_t}D_2Y(X_t,U_t)}{g(D_2Y,D_2Y)} dt\\ - \sum_{l = 1}^d {D_1Y\sigma_l(X_t,U_t)}\dfrac{g^\sharp_{U_t}D_2Y(X_t,U_t)}{g(D_2Y,D_2Y)} dW^l_t.
\end{multline}
\end{subequations}
This method is summarized in algorithm \ref{alg:bounded-sol_manifold}.
{
\begin{algorithm}
\hrule \vspace{1em}
{Numerical computation of bounded m-solution when $D_2Y(X_t,U_t) \not = 0$}\\
\label{alg:bounded-sol_manifold}
\hrule
\begin{algorithmic}[1]
\STATE{If $U_0$ is not available, then use gradient descent to compute $U_0$ such that $Y(X_0,U_0) = 0$;}
\STATE{$b = 1$;}
\STATE {$t_0 = 0$;}
\STATE{Choose $\delta t> 0$;}
\FOR{$i = 1$ \TO $i = T/\delta t$}
\STATE{Solve the SDE  \eqref{equation:numerical_bounded-m-sol} upto time $\delta t$ with initial condition as $(X_{t_{i-1}},U_{t_{i-1}})$;}
\STATE{$t_i = t_{i-1} + \delta t$;}
\WHILE{$\Delta_p(h(X_{t_i}))>\epsilon$}
\STATE{$b = 2b$;}
\STATE{Solve the SDE  \eqref{equation:numerical_bounded-m-sol} upto time $\delta t$ with initial condition as $(X_{t_{i-1}},U_{t_{i-1}})$;}
\STATE{$t_i = t_{i-1} + \delta t$;}
\ENDWHILE
\ENDFOR
\end{algorithmic}
\hrule
\end{algorithm}
}	
Of course, the method fails when $D_2Y(X_t,U_t) = 0$, in which case we will return to the using gradient descent in pseudo time at every time step. However, in this case we will not know if the resulting solution is a semi-martingale or not. Nonetheless, at every time step we will have a point $U_t\in N$ that ensures $Y(X_t,U_t) = 0$. We will call such a solution as \textit{approximate bounded m-solution}. This can be summarized in form of algorithm \ref{alg:approximate-bounded-sol_manifold}.
{
\begin{algorithm}
\hrule \vspace{1em}
{Numerical computation of approximate bounded m-solution.}\\
\label{alg:approximate-bounded-sol_manifold}
\hrule
\begin{algorithmic}[1]
\STATE{If $U_0$ is not available, then use gradient descent to compute $U_0$ such that $Y(X_0,U_0) = 0$;}
\STATE{$b = 1$;}
\STATE {$t_0 = 0$;}
\STATE{Choose $h> 0$;}
\STATE{Choose a positive integer $N$;}
\STATE{$\delta t = N*h$;}
\FOR{$i = 1$ \TO $i = T/\delta t$}
\STATE{$(X_{t^{0}},U_{t^{0}}) = (X_{t_{i-1}},U_{t_{i-1}})$;}
\FOR{$j = 1$ \TO $j = N$}
\IF{$D_2Y(X_{t^0},U_{t^0})\not =0$}
\STATE{Numerical time stepping for SDE  \eqref{equation:numerical_bounded-m-sol} with initial condition as $(X_{t^{j-1}},U_{t^{j-1}})$ and time step of $h$;}
\STATE{$t^j = t^{j-1} + h$;}
\ELSE
\STATE{Numerical time stepping for SDE  \eqref{equation:numerical_bounded-m-sol:SDEfor X} with initial condition as $X_{t^{j-1}}$, time step of $h$, and constant value of $U_t = U_{t^{j-1}}$.}
\STATE{$t^j = t^{j-1} + h$;}
\STATE{Use gradient descent to compute $U_{t^j}$ such that $Y(X_{t^j},U_{t^j}) = 0$, with initial guess as $U_{t^{j-1}}$;}
\ENDIF
\ENDFOR
\STATE{$t_i = t_{i-1} + \delta t$;}
\STATE{$(X_{t_{i}},U_{t_{i}}) = (X_{t^{N}},U_{t^{N}})$;}
\WHILE{$\Delta_p(h(X_{t_i}))>\epsilon$}
\STATE{$b = 2b$;}
\STATE{$(X_{t^{0}},U_{t^{0}}) = (X_{t_{i-1}},U_{t_{i-1}})$;}
\FOR{$j = 1$ \TO $j = N$}
\IF{$D_2Y(X_{t^0},U_{t^0})\not =0$}
\STATE{Numerical time stepping for SDE  \eqref{equation:numerical_bounded-m-sol} with initial condition as $(X_{t^{j-1}},U_{t^{j-1}})$ and time step of $h$;}
\STATE{$t^j = t^{j-1} + h$;}
\ELSE
\STATE{Numerical time stepping for SDE  \eqref{equation:numerical_bounded-m-sol:SDEfor X} with initial condition as $X_{t^{j-1}}$, time step of $h$, and constant value of $U_t = U_{t^{j-1}}$.}
\STATE{$t^j = t^{j-1} + h$;}
\STATE{Use gradient descent to compute $U_{t^j}$ such that $Y(X_{t^j},U_{t^j}) = 0$, with initial guess as $U_{t^{j-1}}$;}
\ENDIF
\ENDFOR
\STATE{$t_i = t_{i-1} + \delta t$;}
\STATE{$(X_{t_{i}},U_{t_{i}}) = (X_{t^{N}},U_{t^{N}})$;}
\ENDWHILE
\ENDFOR
\end{algorithmic}
\hrule
\end{algorithm}
}

\subsection{Example}
\begin{figure}[h]
\includegraphics[width=\linewidth]{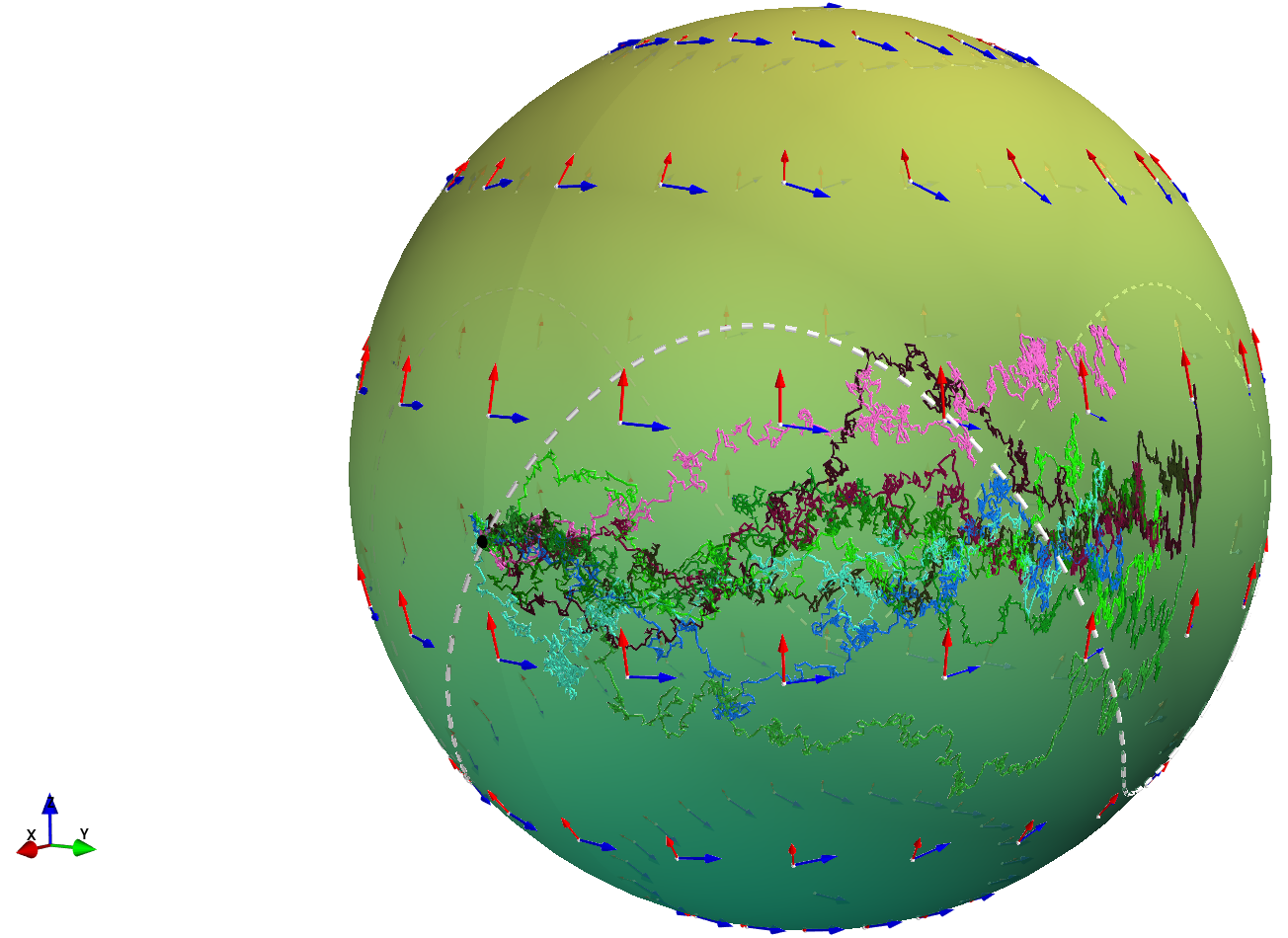}
\caption{The blue colored vector field represents the vector field $i^*K_2$, and the red  vector field represents $i^*K_1$. White line shows the constraint $h = 0$, and the colored lines are the plots for different realizations for the SDE given by equation \eqref{equation:example:_SDE_for_SDAE} when the algebraic variable $u = 0$ and initial condition $(1,0,0)$ (represented by black dot).}
\label{fig:SDAE_withoutcontrol}
\end{figure}
Let us consider a unit sphere $S^2$ embedded in $\mathbb{R}^3$ with center as $0\in \mathbb{R}^3$. We will consider two Wiener processes as the noise drivers. Hence, we need 3 vector fields and a diffusion generator to describe an SDE on the sphere. For the diffusion generator, we will consider standard Ito SDEs on the sphere, which means that we need Ito diffusion generator $G_I$ on the sphere. As $S^2$ is embedded in $\mathbb{R}^3$, we can obtain the vector fields using the pull-back by the embedding function, i.e. if $i:S^2\hookrightarrow \mathbb{R}^3$ is the embedding then any vector field $K\in \mathfrak{X}(Img(i))$ induces a vector field $i^*K\in\mathfrak{X}(S^2)$.

We will consider two vector fields $\sigma_1, \sigma_2 \in \mathfrak{X}(S^2)$ such that $\sigma_j = 0.3i^*K_j$ for $j\in\{1,2\}$. We will consider 
\[K_1(x) = P_x(0,0,1), \&\]
\[K_2(x) = P_x(0,1,0),\]
where $P_x:\mathbb{R}^3\to \mathbb{R}^3$ such that $P_x\circ P_x = P_x$ and $Img(P_x) = Img(T_x i)$.
For the drift coefficient, we will consider $V:M\times TM\to TM$ such that \[V(x,u) = 2i^*K_2(x) + u(i^*K_1(x))\] for all $x\in M$ and $u\in \mathbb{R}$.

\begin{figure}[h!]
\begin{center}
\includegraphics[width=\linewidth]{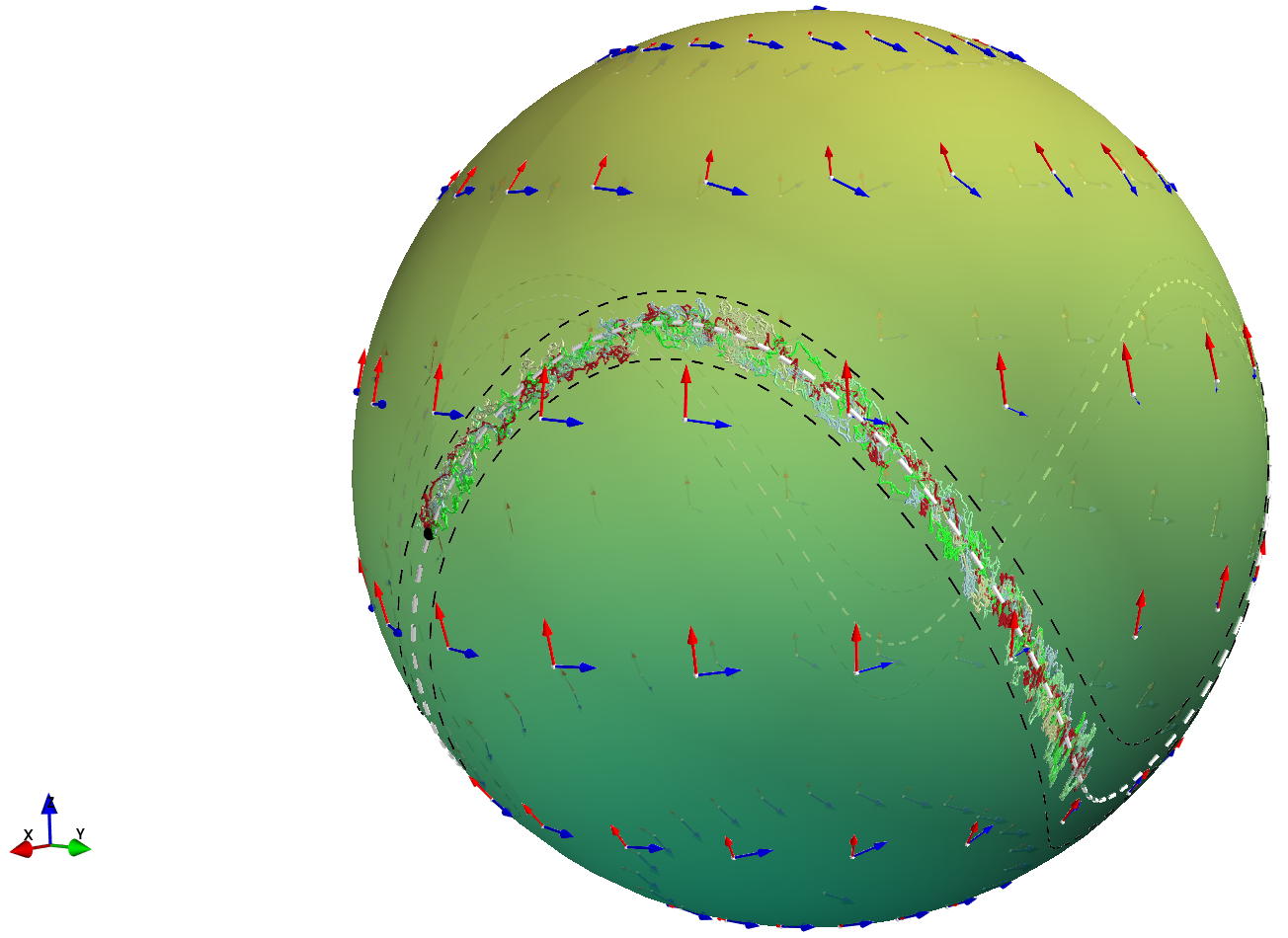}
\end{center}
\vspace{-1em}
\caption{The blue colored vector field represents the vector field $i^*K_2$, and the red  vector field represents $i^*K_1$. White dashed line shows the constraint $h = 0$. Black dashed line represent bound of $\epsilon = 0.1$ around the constraint. The colored lines are the plots for different realizations of the bounded m-solution given as per algorithm \ref{alg:bounded-sol_manifold}.}
\label{fig:SDAE_withcontrol0.1}
\end{figure}

We will consider the constraint as $h:S^2\to \mathbb{R}$ such that in local chart obtained using stereographic projection from $(0,0,1)$,
\[h(X,Y) = 1 + \dfrac{3YX^2 - Y^3}{(X^2 + Y^2)^{3/2}} - \sqrt{X^2 +Y^2} = 0,\]
where $(X,Y)$ is the local coordinates in the stereographic projection.
With this constraint, the SDAE is given as
\begin{subequations}
\begin{equation}
\label{equation:example:_SDE_for_SDAE}
\mathbf{d}x = \left[ 2i^*K_2 + u(i^*K_1(x)) + \dfrac{1}{2}G_I(0.3 i^*K_1) + \dfrac{1}{2}G_I(0.3i^*K_2)\right]dt + \sum_{l = 1}^2 0.3i^*K_l dW^l_t,
\end{equation}
\begin{equation}
h(x) = 0.
\end{equation}
\end{subequations}
In figure \ref{fig:SDAE_withoutcontrol}, we have shown the plots for the unconstrained SDE and the constraint equation.

As the target manifold for the constraint is $\mathbb{R}$, the problem of cut-locus does not arise. From proposition \ref{proposition:bounded m-solution}, we know that
the solution is given by $(X_t,U_t)$ such that it satisfies $Y(X_t,U_t) = 0$, where
\[Y(x,u) = b h(x) + \left(dh\cdot V(x,u)+ \dfrac{1}{2} \sum_{l = 1}^2 dh\cdot \nabla^{G_I}_S(\sigma_l)(x) + \dfrac{1}{2} \sum_{l = 1}^2{}^\mathbb{R}G_S(dh\cdot \sigma_l)(x) \right),\]
\[= b h(x) + dh(x)\cdot V(x,u)+ \dfrac{1}{2} \sum_{l = 1}^2 G_I(\sigma_l)(x)[h(x)].\]
Therefore, 
\begin{equation}
\label{equation:example_solution}
u = \dfrac{- bh(x) - 2dh\cdot i^*K_2(x)- \dfrac{1}{2} \sum_{l = 1}^2 G_I(0.3i^*K_l)(x)[h(x)]}{dh\cdot i^*K_1(x)},
\end{equation}
will ensure that $Y(x,u) =0$. The approximate solution obtained using this solution is shown in figure \ref{fig:SDAE_withcontrol0.1} for $\epsilon = 0.1$.

\section{Concluding remarks}
In this article we have extended the ideas of Stochastic Differential-Algebraic Equations (SDAEs) from the Euclidean spaces to the setting of smooth manifolds. We have observed that, like SDAEs on Euclidean spaces, SDAEs on manifolds can be classified into index 1 SDAE and high index SDAE. The necessary condition for existence of the solution of the SDAE gave a class of SDAE, that we have called as ill-posed, for which there is no solution. This classification is shown in the following flowchart.

\begin{center}
\begin{tikzpicture}[node distance=2cm]
\node (start) [startstop] {{\textbf{SDAEs on Manifolds}}};
\node (base0) [process, below of=start, xshift=-3.5cm] {\textbf{Index 1 SDAEs} (Derivative of constraint with algebraic variable is invertible)};
\node (base1) [process, below of=start, xshift=+3.5cm] {\textbf{High index SDAEs} (Derivative of constraint with algebraic variable is not invertible)};
\node (pro1) [process, below of=base1, xshift=-6cm] {\textbf{Completely high index SDAEs} (constraint is not a function of algebraic variable)};
\node (pro2) [process, below of=base1, xshift=+1cm] {High index SDAEs not completely high index (constraint is a function of algebraic variable, but it is not index 1)};
\node (last1) [process, below of=pro1, xshift=-1.5cm] {\textbf{Ill-posed} (no solution)};
\node (last2) [process, below of=pro1, xshift=+5.5cm] {Not ill-posed};
\draw [arrow] (start) -- (base0);
\draw [arrow] (start) -- (base1);
\draw [arrow] (base1) -- (pro1);
\draw [arrow] (base1) -- (pro2);
\draw [arrow] (pro1) -- (last1);
\draw [arrow] (pro1) -- (last2);
\end{tikzpicture}
\end{center}

We have shown that it is possible to approximately solve completely high index SDAE, ill-posed or not. The approximate solutions are obtained by relaxing the constraint. If $h:M\to P$ is the constraint function such that we require $h(X_t) = p$, then the first way of relaxing the constraint is by allowing the distance of $h(X_t)\in P$ from point $p\in P$ to be less than or equal to some $\epsilon>0$. We have presented a technique for finding such approximate solution. However, this technique relied on construction of a function $Z:N\to P$, which is not easy to construct on manifolds. For this reason we find that it is convenient to relax the constraint further by allowing the approximate constraint be violated upto a bounded probability. We have presented numerical algorithm for computing such approximate solution. These numerical methods are similar to the one presented in \cite{suthar2021explicit} for SDAEs on Euclidean spaces. We find that with the diffusion generator approach of describing the SDEs on manifolds, all the results for SDAEs on Euclidean spaces from \cite{suthar2021explicit}, translate to the manifold settings with very little modifications.

\bibliography{3}
\end{document}